\documentclass[11pt,a4paper]{article}
\usepackage[OT1]{fontenc}
\usepackage[latin1]{inputenc}
\usepackage[english]{babel}
\usepackage{amsfonts}
\usepackage{amsthm}
\newtheorem{teo}{Theorem}[section]
\newtheorem{prop}[teo]{Proposition}
\newtheorem{lem}[teo]{Lemma}
\newtheorem{coro}[teo]{Corollary}

\theoremstyle{definition}
\newtheorem{rem}[teo]{Remark}

\begin{document}

\newcommand\grass{{\cal P}(M_1)}
\newcommand\p{{\sf \,p\,}}
\newcommand\h{{\cal H}}
\newcommand\B{{\cal B}}
\newcommand\um{ {\cal O}(\p)  }
\newcommand\hc{{\cal H}_{\mathbb C}}
\newcommand\hr{{\cal H}_{\mathbb R}}
\newcommand\hsh{{\sf HS^h}}
\newcommand\be{{\cal B}_2({\cal H})}
\newcommand\bpe{{\cal B}_p({\cal H})}
\newcommand\bqe{{\cal B}_q({\cal H})}
\newcommand\beh{{\cal B}_2({\cal H})_h}
\newcommand\beah{{\cal B}_2({\cal H})_{ah}}
\newcommand\gldos{Gl_2({\cal H})}
\newcommand\glpe{Gl_p({\cal H})}
\newcommand\udos{U_2({\cal H})}
\newcommand\upe{U_p({\cal H})}
\newcommand\pei{{\left<\right.}}
\newcommand\ped{{\left>\right.}}
\newcommand\gl{Gl({\cal H})}
\newcommand\oA{ {\cal O}_A}
\newcommand\ga{ {G}_A}
\newcommand\gp{ {G}_p}
\newcommand\oo{ {\cal O} }
\newcommand\f{ {\cal F} }
\newcommand\g{ {\cal G} }
\newcommand\ad{ {\mbox{ad\,}} }

\title{\vspace*{0cm}Finsler geometry and actions of the p-Schatten unitary groups}
\date{}
\author{Esteban Andruchow, Gabriel Larotonda and L\'azaro Recht}

\maketitle

\abstract{\footnotesize{\noindent Let $p$ be an even positive integer and $\upe$ be the Banach-Lie group of unitary
operators $u$ which verify that $u-1$ belongs to the $p$-Schatten ideal $\bpe$. Let $\oo$ be a smooth manifold on which $\upe$ acts transitively and smoothly. 
Then one can  endow $\oo$ with a natural  Finsler metric in terms of the $p$-Schatten norm and the action of $\upe$. Our main result establishes that for any pair of given initial conditions
$$
x\in \oo\hbox{ and } X\in(T\oo)_x
$$
there exists a  curve $\delta(t)=e^{tz}\cdot x$  in $\oo$, with $z$ a skew-hermitian element in the $p$-Schatten class such that
$$
\delta(0)=x \hbox{ and } \dot{\delta}(0)=X,
$$
which remains minimal as long as $t\|z\|_p\le \pi/4$. Moreover, $\delta$ is unique with these properties. We also show that the metric space $(\oo,d)$ ($d=$ rectifiable distance) is complete. In the process we establish minimality results in the groups $\upe$, and a convexity property for the rectifiable distance.
As an example of these  spaces, we treat the case of the unitary orbit 
$$
{\cal O}=\{uAu^*: u\in \upe\}
$$
of a self-adjoint operator $A\in B(\h)$.
}}

\setlength{\parindent}{0cm} 

\section{Introduction}

\bigskip

Let $\h$ be an infinite dimensional Hilbert space and ${\cal B}(\h)$ be the space
of bounded linear operators acting in $\h$. Denote by $\bpe$ the $p$-Schatten class
$$
{\cal B}_p(\h)=\{ a\in {\cal B}(\h): Tr((a^*a)^{p/2})<\infty\}.
$$
where $Tr$ is the usual trace in ${\cal B}(\h)$.
In this paper we shall focus on the case when $p$  is an even integer. 
The spaces ${\cal B}_p(\h)$ are Banach spaces with the norms
$$
\|a\|_p=Tr((a^*a)^{p/2})^{1/p}.
$$
We use the subscript $h$ (resp. $ah$) to denote the sets of hermitian (resp. skew-hermitian) operators, e.g. $\bpe_h=\{x\in\bpe: x^*=-x\}$. Throughout this paper, $\| \ \|$ denotes the usual operator norm.
Denote by ${\cal G}l(\h)$ the linear group and by ${\cal U}(\h)$ the unitary group of $\h$.
Consider the following classical Banach-Lie groups groups of operators \cite{harpe}:
$$
\glpe=\{g\in \gl: g-1\in {\cal B}_p(\h)\},
$$
and
$$
\upe=\{u\in U(\h): u-1\in {\cal B}_p(\h)\},
$$
where $1\in {\cal B}(\h)$ denotes the identity operator. These groups have differentiable structure when endowed with the
metric $\|g_1-g_2\|_p$ (note that $g_1-g_2\in\bpe$). For instance, the Banach-Lie algebra of $\upe$ is the (real) Banach space $\bpe_{ah}$.

Let $\oo$ be a topological space on which $\upe$ acts transitively, such that for any element $x\in\oo$, the subgroup
$G_x=\{u\in\upe: u\cdot x=x\}$ is a closed submanifold of $\upe$. This implies that $\oo$ can be endowed with a differentiable manifold structure, in a way such that the map
$$
\pi=\pi_{x}: \upe \to \oo, \ \ \pi_{x}(u)=u\cdot x
$$
is a smooth submersion. In other words, $\oo\simeq \upe / G_x$ is  a  smooth homogeneous space of the group $\upe$. The main object of this paper is the geometric study of this space, under reasonably  general conditions, which are specified below. We introduce a Finsler metric $\{\|\ \|_x: x\in\oo\}$ in $\oo$,  (a Riemannian metric if $p=2$) induced by the $p$ norm in $\bpe$ and by the action. We focus on the existence of metric geodesics, i.e. curves of minimal length. Our approach is to study the metric geometry of the group
$\upe$ in order to obtain results in $\oo$. In the process we find properties in $\upe$ which we claim are interesting in their own right. For instance:
\begin{enumerate}
\item
The one-parameter unitary groups $e^{tz}\in\upe$ ($z\in\bpe_{ah}$), regarded as curves of unitaries, have minimal length in the $p$-norm, as long as $t\|z\|\le \pi$ (note that this  condition is given in terms of the usual norm $\|z\|$ of $z$, a fact that implies that there are arbitrarily long minimal curves in $\upe$).
\item
The map $f_p(t)=d_p(u_0,e^{tz})^p$, where $d_p$ is the rectifiable metric induced by the $p$-norm, and $u_0$ is a fixed element in $\upe$, is a strictly convex function, provided that $u_0$ and the endpoints of the curve lie at distance not greater than $\pi/4$. 
\end{enumerate}
Denote by $\g_x$ the Banach-Lie algebra of $G_x$.
We shall make the assumption that $G_x$ is locally exponential: since any element $u\in\upe$ is of the form $u=e^z$ for some $z\in {\cal B}_p(\h)$, we ask that for any element $v\in G_x$ close to $1\in G_x$, there exists an element $z\in \g_x$ such that $v=e^z$. Apparently, if this holds for a given $x_0\in\oo$, then it holds for any $x\in \oo$ (since the groups $G_x$ and $G_{x_0}$ are conjugate by an inner automorphism).

Using these facts we prove our main results on $\oo$:
\begin{enumerate}
\item
If $x\in\oo$ and $X\in(T\oo)_x$, then there exists a unique curve   $\gamma(t)=e^{tz}\cdot x$ with $\gamma(0)=x$ and $\dot{\gamma}(0)=X$, which has minimal length in $\oo$ as long as $t\|X\|_x\le \pi/4$.
\item
The metric space $(\oo,d)$ is complete, where $d$ is the rectifiable metric induced by the Finsler metric in $\oo$.
\end{enumerate}

There are many examples of this situation. For instance, if $A\in {\cal B}(\h)$ is a self-adjoint operator, its unitary orbit $\oA=\{uAu^*: u\in\upe\}$ is a homogeneous  space, the group $G_A$ consists  of the elements of $\upe$ which commute with $A$. 
$\ga$ is a Banach-Lie subgroup of $\upe$ since it is an algebraic subgroup (cf. Theorem 4.13 in \cite{beltita}), and its Lie algebra is given by
$$
{\cal G}_A=\{ x\in \bpe_{ah} : xA-Ax=0\}.
$$

Unitary orbits of operators have been studied before from a geometric point of view in
\cite{andsto, herrero2, belrattum, bona,cprprojections,larotonda,pr}. In this particular framework, restricting the action to
these classical groups $\upe$, certain results can be found in  \cite{belrattum,bona,carey,larotonda}.

\bigskip
Let us briefly describe the contents of the paper. In Section 2 we introduce the Finsler metric which is Riemannian if $p=2$. In Section 3 we examine the metric structure of the group $\upe$ endowed with the Finsler metric given by the $p$-norm. We recall certain known facts, and prove  results which we believe are new, among them the two results described above. In Section 4 we show the consequence of these facts on the homogeneous space $\oo$: existence and uniqueness of short curves with given initial data.  In Section 5 we prove that the metric spaces $\oo$ are complete. Section 6 is devoted to the example $\oo_A$, $p=2$, giving a characterization of the case when $\oo_A$  is a smooth submanifold of the affine Hilbert space $A+ \be_h$. In Section 7 we state what we believe is the main open problem in this setting, namely the existence of minimal curves joining given endpoints in $\oo$, and prove a partial positive result.

\section{Linear connections and metrics}
Let us first consider the case $p=2$.
One can induce a metric in (the tangent spaces of) $\oo$ by means of the decomposition
$$
\beah={\cal G}_{x}\oplus {\cal F}_x,
$$
where ${\cal F}_x$ is the $Tr$-orthogonal complement of ${\cal G}_{x}$. Apparently, $\f_x$  is invariant by the inner action of $G_{x}$. Therefore this decomposition defines what in classical geometry of homogeneous spaces \cite{sharpe} is called a Reductive Structure. 

The kernel of $d(\pi_{x})_1$ is $\g_{x}$, therefore 
$$
\delta_{x}:=d(\pi_{x})_1|_{\f_x}:\f_x \to (T\oo)_{x}
$$
is a linear isomorphism. Denote by $\kappa_{x}$ its inverse, and by $P_x$ the $Tr$-orthogonal projection 
$$
P_x: \beah\to \f_x\subset \beah.
$$

We  endow $(T\oo)_{x}$ with the following inner product
\begin{equation}\label{metrica}
<V,W>_{x}=Tr(\kappa_{x}(W)^*\kappa_{x}(V))=-Tr(\kappa_{x}(W)\kappa_{x}(V)) , \ \ V,W\in(T\oo)_{x}.
\end{equation}
Clearly the distribution $x\mapsto <\ ,\ >_x$ is smooth, in the sense that that if $V,W$ are tangent fields in $\oo$, then the map $\oo\ni x\mapsto <V_x,W_x>_x$ is smooth, and therefore (\ref{metrica}) defines a Riemann-Hilbert metric in $\oo$. 

The Levi-Civita connection of this metric can be computed. In the paper  \cite{mr} two natural linear connections for a homogeneous reductive space were introduced. The first, which is called the reductive connection $\nabla^r$, is the analogous to the connection that one obtains for a reductive manifold in finite dimensions. It can be described as follows. If $V$ is a tangent field  and  $W$ is a tangent vector (at $x$) in $\oo$, then
$$
\kappa_x (\nabla^r_W V(x))=\kappa_x(W)( \kappa_x (V_x))+[\kappa_x(V_x), \kappa_x(W)],
$$
where $[\ , \ ]$ is the commutator of operators in ${\cal B}(\h)$, and  $a(b)$ denotes the the derivative of $b$ in the direction of $a$.

A straightforward computation shows that since the maps $\kappa$ are isometric, the reductive connection is compatible with the metric defined.

The second natural connection for a reductive space is the {\it classifying } connection  $\nabla^c$. 
Suppose   $V, W$ are as above, then 
$$
\nabla^c_V(W)(x)=\delta_x P_x(\kappa_x(V)[\kappa_x(W)]_x)
$$

These two connections share the same geodesics, which are described below, and have torsion tensors with opposite signs. It follows that the connection 
$$
\nabla=\frac12(\nabla^r+\nabla^c)
$$ 
has zero torsion, and the same geodesics. 
We claim that this connection $\nabla$ is the Levi-Civita connection of the
metric (\ref{metrica}) introduced above, in the sense  that it is symmetric
(torsion free) and compatible with the metric. To prove this claim, it only remains to show that $\nabla^c$ is compatible with the metric.
\begin{lem}
The classifying connection $\nabla^c$ is compatible with the
metric $<\ , \ >_x$ in $\oo$.
\end{lem}
\begin{proof}
Let $V(t)$, $W(t)$ be two tangent fields along the
curve $\nu(t)$ in $\oo$.
Then
$$
<\frac{D^c V}{d t},W>_\nu=-Tr(\kappa_\nu(W)\kappa_\nu(\frac{D^c V}{d t}))=-Tr(\kappa_\nu(W)P_\nu(\dot{\kappa_\nu(V)})).
$$
Note that since $\kappa_\nu(W)\in R(P_\nu)$, $Tr(\kappa_\nu(W)P_\nu(\dot{\kappa_\nu(V)}))=Tr(\kappa_\nu(W)\dot{\kappa_\nu(V)})$.
Analogously
$$
< V,\frac{D^c W}{d t}>_\nu=-Tr(\dot{\kappa_\nu(W)}\kappa_\nu(V)).
$$
Then 
\begin{eqnarray}
<\frac{D^c V}{d t},W>_\nu+< V,\frac{D^c W}{d t}>_\nu & =  & -Tr(\kappa_\nu(W)\dot{\kappa_\nu(V)})-Tr(\dot{\kappa_\nu(W)}\kappa_\nu(V))\nonumber\\
& =& \frac{d}{dt}<V,W>_\nu.\nonumber
\end{eqnarray}
\qedhere
\end{proof}

The geodesics of these connections are computed explicitly in \cite{mr}. For instance, the geodesic $\gamma$ with $\gamma(0)=x$ and $\dot{\gamma}(0)=V$ is given by
$$
\gamma(t)=e^{\kappa_x(V)}\cdot x  , \ t\in \mathbb{R}.
$$
In other words, geodesics of $\oo$ are of the form $e^{tz}\cdot x$, for $z\in\f_x$.

The following linear differential equation is usually called the {\it horizontal lifting} equation of the reductive structure: 
\begin{equation}\label{levantamientohorizontal}
\left\{
\begin{array}{l}
\dot{\Gamma}= \kappa_\gamma(\dot{\gamma}) \Gamma \\
\Gamma(0)=  1 .
\end{array}
\right.
\end{equation}
It is a linear differential equation in $\be$. In order to assure the existence and uniqueness of solutions, one must check that the mapping 
$$
[0,1]\ni t \mapsto \kappa_{\gamma(t)}(\dot{\gamma}(t))\in \beah,  
$$
is smooth. This is clear if $\gamma$ is smooth. 

Therefore the equation (\ref{levantamientohorizontal}) has, for a given $\gamma$, a unique solution. One can prove, as in classical homogeneous reductive spaces \cite{sharpe}, the following result.
\begin{prop}
Let $\gamma(t)$, $t\in[0,1]$ be a smooth curve in $\oo$. Then the unique solution $\Gamma$  of (\ref{levantamientohorizontal}) verifies
\begin{enumerate}
\item
$\Gamma(t)\in\udos, \ t\in[0,1].$
\item
$\Gamma$ lifts $\gamma$: $\pi_{\gamma}(\Gamma)=\gamma$.
\item
$\Gamma$ is horizontal: $\Gamma^*\dot{\Gamma} \in \f_\gamma$.
\end{enumerate}
\end{prop}

It will be useful to take a brief look at the natural Riemannian geometry of the group $\udos$. Namely,
the metric given by considering the trace inner product, and
therefore, the $2$-norm at each tangent space. The tangent spaces of $\udos$
are
$$
(T \udos)_u=u \beah= \beah u.
$$
The covariant derivative consists of differentiating in the
ambient space, and projecting (orthogonally with respect to the real part of the trace) onto $T\udos$. Geodesics of the Levi-Civita connection are curves of the form
$$
\mu(t)=u e^{tx},
$$
for $u\in \udos$ and $x\in\beah$.
The exponential mapping of this connection is the  map
$$
exp: \beah\to \udos , \ \ exp(x)=e^x.
$$

In the general case $p>2$, one can endow the homogeneous space $\oo$ with a Finsler metric, derived from the $p$-norm and the group action. Following ideas in \cite{duranmatarecht}, we shall not  consider a linear connection in this case, and focus only on characterizing short curves (or metric geodesics), which are not the geodesics of any linear connection. First let us introduce some notation. The action of $\upe$ on $\oo$ induces two kind of maps. If one fixes $x\in\oo$, one has the submersion
$$
\pi_x:\upe\to \oo, \ \ \pi_x(u)=u\cdot x , \ \ u\in\upe.
$$
If one fixes $u\in\upe$ one has the diffeomorphism
$$
\ell_u:\oo\to \oo,\  \ \ell_u(x)=u\cdot x, \ \ x\in\oo.
$$
If $x\in\oo$ and $X\in(T\oo)_x$,  put
$$
\|X\|_x=\inf\{ \|z\|_p: z\in\bpe_{ah}, (d\pi_x)_1(z)=X\}.
$$
This metric could be called the quotient metric of $\oo$, because it is the quotient metric in the Banach space $(T\oo)_x$ if one identifies it with $\bpe/\g_x$. 
Indeed, since $\g_x=\ker (d\pi_x)_1$, if $z\in \bpe_{ah}$ with $(d\pi_x)_1(z)=X$, then
$$
\|X\|_x=\inf \{ \|z-y\|_p: y\in\g_x\}.
$$
Note that if $p=2$, this metric coincides with the previously defined Riemannian metric. Indeed, if $Q_x=1-P_x$ is the orthogonal projection onto $\g_{x}$, then each $z\in{\cal B}_2({\cal H})_{ah}$ can be uniquely decomposed as
$$
z=z-Q_x(z)+Q_x(z)=z_0+Q_x(z),
$$
hence
$$
\|z-y\|_2^2=\|z_0+Q_x(z)-y\|_2^2=\|z_0\|_2^2+\|Q_x(z)-y\|_2^2\ge \|z_0\|_2^2
$$
for any $y\in \g_{2,x}$, which shows that
$$
\|X\|_x=\inf \{ \|z-y\|_2: y\in\g_{x}\}=\|z_0\|_2,
$$
where $z_0$ is the unique vector in $\g_{x}^{\perp}$ such that $(d\pi_x)_1(z_0)=X$.

One of the main features of this metric in $\oo$ is that it is invariant by the group action (or in other words, that the group acts isometrically on the tangent spaces):
if $x\in\oo$, $X\in T_x\oo$ and $u\in\upe$,
$$
\|(\ell_u)_{*x}(X)\|_{u\cdot x}=\|X\|_x.
$$
Indeed, $\pi_{u \cdot x}= \pi_x \circ R_u$, where $R$ denotes the right product in $\upe$. Then $(\pi_{u \cdot x} )_{* \, 1}= (\pi_x)_{* \, u} \circ R_u$. On the other hand, $\pi_x = \ell_u \circ \pi_x \circ L_{u^*}$, where $L_{u^*}$ denotes the left product, so $(\pi_x)_{* \, u}= (\ell_u)_{* \, x} \circ (\pi_x)_{* \, 1} \circ L_{u^*}$. Then
$$
(\pi_{u \cdot x})_{* \,1} = (\ell_u)_{* \, x} \circ (\pi_x)_{* \, 1} \circ Ad_{u^*}. 
$$
Hence $X=(\pi_x)_{* \, 1}(z)$ if and only if $(\pi_{u \cdot x})_{* \, 1}(uzu^*)= (\ell_u)_{* \, x}(X)$. Since the $p$-norms are unitarily invariant,
$$
\|(\ell_u)_{* \, x}(X)\|_{u\cdot x} = \|X\|_x.
$$ 

\smallskip

Throughout, $L$ denotes the length functional for piecewise smooth
curves in $\oo$, measured with the quotient norm introduced above,
$$
L(\gamma)=\int_{t_0}^{t_1} \|\dot{\gamma}(t)\|_{\gamma(t)}\, d t
$$
and $d$ the rectifiable distance in $\oo$:
$$
d(x_1,x_2)=\inf\{ L(\gamma): \gamma \subset \upe \hbox{ joins } x_1 \hbox{ and } x_2 \}. 
$$

\section{Metric structure of $\upe$}

In this section we recall and complete certain facts from \cite{odospe}, concerning the minimality of geodesics in $\upe$. Afterwards we establish local convexity results for the geodesic distance. These results will be the key to obtain minimality results in $\oo$. Proofs for these statements for the case $p=2$ can be found in \cite{odospe}.

\bigskip

Throughout this paper, $L_p$ denotes the length functional for piecewise smooth
curves in $\upe$, measured with the $p$-norm:
$$
L_p(\alpha)=\int_{t_0}^{t_1} \|\dot{\alpha}(t)\|_p\, d t
$$
and $d_p$ the rectifiable distance in $\upe$:
$$
d_p(u_1,u_2)=\inf\{ L_p(\gamma): \gamma \subset \upe \hbox{ joins } u_1 \hbox{ and } u_2 \}. 
$$

\begin{rem}\label{remarko}
\begin{enumerate}
\item
The exponential map 
$$
exp: \bpe_{ah}\to \upe
$$
is surjective. 
\item
The exponential map is a bijection between the sets
$$
\bpe_{ah}\supset \{z\in\bpe_{ah}: \|z\|<\pi\}\to \{u\in\upe: \|1-u\|<2\}.
$$
\item
Moreover, 
$$
exp:\{z\in \bpe_{ah}: \|z\|\le \pi\}\to \upe,
$$
is surjective.
\end{enumerate}

\medskip

These facts can be obtained from the following observation. If $u\in \upe$, then it has a spectral decomposition $u=p_0+\sum_{k\ge 1} (1+\alpha_k)p_k$, where $\alpha_k$ are the non zero eigenvalues of $u-1\in\bpe$. There exist $t_k\in {\mathbb R}$ with $|t_k|\le \pi$ such that $e^{it_k}=1+\alpha_k$. The elementary estimate
$$
|t_k|^p(1-\frac{|t_k|^2}{12})^{p/2}\le |e^{it_k}-1|^p=|\alpha_k|^p
$$
implies that the element $z=\sum_{k\ge 1} it_kp_k$, whose exponential is $u$, lies in $\bpe_{ah}$.
\end{rem}

The following result  states that the one parameter groups of unitaries in  $\upe$ have
minimal length up to a certain critical value of $t$. This could be derived from
the general theory of Hilbert-Riemann manifolds for the case $p=2$. In any case, the proof, which is essentially contained in \cite{odospe}, is  operator theoretic,  and  provides a uniform lower bound for the geodesic radius.

\begin{teo}\label{minimalidadunitarios}
The following facts hold.
\begin{enumerate}
\item
Let $u\in\upe$ and $x\in\bpe_{ah}$ with $\|x\|\le \pi$. Then the curve
$\mu(t)=ue^{tx}$, $t\in[0,1]$ is shorter than any other piecewise smooth curve in
$\upe$ joining the same endpoints. Moreover, if $\|x\|<\pi$,  $\mu$ is unique with
this property.

\item
Let $u_0, u_1 \in\upe$. Then there exists a minimal geodesic curve joining them.
If $\|u_0-u_1\|<2$,  this geodesic is unique.
\item
There are in $\upe$ minimal geodesics of arbitrary length. Thus the diameter of $\upe$ is infinite.
\item
If $u,v\in\upe$ then 
$$
\sqrt{1-\frac{\pi^2}{12}} \; d_p(u,v) \le \|u-v\|_p\le d_p(u,v).
$$
In particular the metric space $(\upe, d_p)$ is complete.
\end{enumerate}
\end{teo}
\begin{proof}
Concerning the first statement, in \cite{odospe} the following was proved. If $u\in\udos$ and $x\in\beah$ with $\|x\|\le \pi$, then the curve $\mu(t)=ue^{tx}$ is minimal for $t\in[0,1]$, when the length is measured with the $p$-norm. Clearly it suffices to treat the case $u=1$. Suppose that there exists a curve $\gamma(t)\in\upe$ with $L_p(\gamma)<L_p(\mu)+\epsilon$.  One can approximate $x$ with a skew-hermitian operator $z$ of finite spectrum with the following properties:
\begin{enumerate}
\item
$\|z\|\le \|x\|\le\pi$.
\item
$\|x\|_p- \epsilon/2<\|z\|_p\le \|x\|_p$.
\item
There exists a $C^\infty$  curve of unitaries joining $e^{x}$ and $e^{z}$, of
$p$-length  less than $\epsilon/2$.
\end{enumerate}
The first two conditions are clear. The third  can be obtained as follows. Put $e^{-x}e^{z}=e^{y}$, with $y\in \bpe_{ah}$. The element $z$ can
be adjusted so as to obtain $y$ of arbitrarily small $p$-norm. Then the curve of
unitaries $\nu(t)=e^{x}e^{ty}$ is $C^\infty$,  joins $e^{x}$ and $e^{z}$
with  $p$-length $\|y\|_p<\epsilon/2$.

Consider now the curve $\gamma_1$, which is the curve $\gamma$ followed by the curve
$e^{x}e^{ty}$  above. Then clearly
$$
L_p(\gamma_1)\le L_p(\gamma)+\|y\|_p<L_p(\gamma)+\epsilon/2 .
$$
Note that $L_p(\gamma_1)< \|x\|_p-\epsilon/2$ and that $\gamma_1$
joins $1$ and $e^{z}$. We claim that there exists a curve $\gamma_2$ in $\udos$, also joining $1$ and $e^z$, with length
$L_p(\gamma_2)<L_p(\gamma_1)+\epsilon/4$. Indeed, the curve $\gamma_1$ is of the form $\gamma_1(t)=e^{\alpha(t)}$ for a continuous piecewise $C^1$  path $\alpha\in\bpe_{ah}$ with endpoints $0$ and $z$. By compactness of the unit interval, one can uniformly approximate $\alpha$ by a curve $\beta$ with the  same endpoints, lying in $\beah$, in order that $\gamma_2(t)=e^{\beta(t)}$ verifies our claim. These facts imply that the curve $\gamma_2$ in $\udos$ which joins $1$ and $e^z$, is shorter than the curve $e^{tz}$ (which lies in $\udos$ because the spectrum of $z$ is finite). This contradicts the minimality statement in $\udos$ proved in \cite{odospe}.

Let us prove that if $\|x\|<\pi$, then $\mu$ is unique with the minimality property.
To do this we shall follow a standard procedure, using the first variation formula for the functional $F_p$ which is given by
$$
F_p(\gamma)=\int_0^1 \|\dot{\gamma}(t)\|_p^p d t,
$$
if $\gamma(t)\in\upe$, $t\in[0,1]$.

Let $\gamma_s(t)$, $t\in [0,1]$, $s\in (-r,r)$ be a smooth variation of the curve $\gamma$, i.e.
\begin{enumerate}
\item
$\gamma_s(t)\in \upe$, for all $s,t$.
\item
The map $(s,t)\mapsto \gamma_s(t)$ is smooth.
\item
$\gamma_0(t)=\gamma(t)$.
\end{enumerate}
We shall use a formula for 
$$
\frac{d}{d s} F_p(\gamma_s)\arrowvert_{s=0}.
$$
obtained in \cite{convexg} in the context of a $C^*$-algebra with trace, which applies here because the formal computations are the same (they only involve partial derivatives and integration by parts).
As in classical differential geometry, we shall call the expression obtained the first variation formula.  Let  
$$
V_s=\frac{d}{d t}{\gamma}_s \ \ \hbox{ and }\ \  W_s=\frac{d}{d s}\gamma_s.
$$
With lower case types we denote the left translations 
$$
v_s=\gamma_s^* V_s \ \ \hbox{ and }\ \  w_s=\gamma_s^*W_s.
$$
 Note that $V_s, W_s\in (T\upe)_{\gamma_s}$ whereas $v_s, w_s\in \bpe_{ah}$.

Then
$$
\frac{(-1)^{p/2}}{p} \frac{d}{d s} F_p(\gamma_s)=Tr(v_s^{p-1}w_s)\arrowvert_{t=0}^{t=1}-\int_0^1 Tr( \frac{d}{d t}[v_s^{p-1}] w_s ) d t.
$$

Suppose that $\gamma(t)\in\upe$ is a smooth minimal curve, and let $\gamma_s(t)$ be a variation, with fixed endpoints $\gamma(0)$ and $\gamma(1)$, i.e. $\gamma_s(0)=\gamma(0)$ and $\gamma_s(1)=\gamma(1)$ for all $s$. Then
$\frac{d}{ds}F_p(\gamma_s)|_{s=0}=0$, and thus
$$
0=Tr(v_0^{p-1}w_0)|_{t=0}^{t=1}- \int_0^1 Tr(w_0 \frac{d}{dt}(v_0^{p-1})) d t.
$$
The fixed endpoints hypothesis implies that the first term vanishes. Then
$$
\int_0^1 Tr(w_0 \frac{d}{dt}(v_0^{p-1})) d=0
$$
for any variation $\gamma_s$ with fixed endpoints.
Let us denote by $Z(t)=\frac{d}{dt}(v_0^{p-1})$ and by $A(t)=w_0(t)$. Both $A$ and $Z$ are continuous fields, $A$ in $\bpe_{ah}$ and $Z$ in $\bqe_{ah}$, where $1/p+1/q=1$. The variation formula implies that
$$
\int_0^1 Tr(A(t)Z(t)) d t = 0
$$ 
for any continuous field  $A$ in $\bpe_{ah}$ such that $A(0)=A(1)=0$.  We claim that this condition implies that $Z(t)=0$ for all $t$. 

First note that the requirement that the field $A$ vanishes at $0$ and $1$ can be removed: let $f_r(t)$ be a real function which is constant and equal to $1$ in the interval $[r, 1-r]$ and such that $f(0)=f(1)=0$, with $0\le f_r(t)\le 1$ for all $t$. Let $B(t)$ be any continuous field in $\bpe_{ah}$ and consider  $A_r(t)=f_r(t)B(t)$. Then $\int_0^1 A_r(t) Z(t) d t =0$, and if $r\to 0$, $\int_0^1 B(t) Z(t) d t =0$. Also it is clear that the integral will vanish if $A$ is non skew-hermitian. Indeed, it is clear if $A$ is hermitian, and for general $A$,  decompose $A$ as the sum of its hermitian and skew-hermitian parts.

Fix $t_0$ in the interval $[0,1]$. Let $Z(t_0)=u|Z(t_0)|$ be the polar decomposition, and consider $x=|Z(t_0)|^{q-1}u^*\in\bpe$. Consider the field $A(t)=xg(t)$ in $\bpe$, with $g$ a convenient smooth support function. Then
$$
0=\int_0^1 Tr(x Z(t))d t\ge c\|Z(t_0)\|_q^q.
$$
Then $v_0^{p-1}$ is constant, and since $p$ is even and $v_0$ is skew-hermitian, $v_0(t)=\gamma(t)^*\frac{d}{d t}\gamma(t)$ is constant, i.e. $\gamma(t)=e^{tx}$ for some $x\in\bpe_{ah}$.

Fact 2. was proved in \cite{odospe}, the (algebraic) argument for $p>2$ is the same as for $p=2$. 

Fact 3. was proved in \cite{odospe}.

Fact 4. follows from the elementary estimate in the remark above.
\end{proof}

Let us establish further facts on the metric structure of the group $\upe$.

\begin{lem}\label{expo}
Let $a,b\in \bpe$, let $\exp:\bpe\to 1+\bpe$ be $\exp(x)=e^x$, and $\ad a:\bpe\to \bpe$ the operator $\ad ax=xa-ax$. Then
$$
d \exp_a(b)=\int\limits_0^1 e^{(1-t)a}be^{ta}\, dt=e^a\,F(\ad a)b= F(\ad a)(e^a\, b),
$$
where $F(z)=\frac{e^z-1}{z}=\sum_{n\ge 0}\frac{z^{n}}{(n+1)!}$. The differential is invertible at $a$ if and only if $\sigma(\ad a)\cap \{2k\pi i\}=\emptyset$ ($k\in \mathbb Z_{\ne 0}$), and then 
$$
d\exp_a^{-1}(w)=e^{-a}F(\ad a)^{-1}w.
$$
In particular if $\|a\|<\pi$ then $d\exp_a$ is invertible. If  $a\in \bpe_{ah}$, then the differential is a contraction:
$$
\|d\exp_a(b)\|_p\le \|b\|_p.
$$
\end{lem}
\begin{proof}
Compute $\lim\limits_{s\to 0}\frac{e^{a+sb}-e^a}{s}$,
applied to the identity
$$
e^{a+b}-e^a=\int\limits_0^1 e^{(1-t)a}be^{t(a+b)}  ,
$$
which is elementary and can be proven integrating by parts the functions
$f(t)=e^{(1-t)a}$ and $g(t)=e^{t(a+b)}$ in $[0,1]$. To prove the second equality, write
$$
e^{-ta}be^{ta}=e^{t(R_a-L_a)}(b),
$$
where $L_a(x)=ax$ and $R_a(x)=xa$ denote left and right multiplication by $a$. Then
$$
\int_0^1 e^{t(R_a-L_a)}\, dt=\sum_{n\ge 0}\frac{1}{n!} \int_0^1 t^n\,dt (R_a-L_a)^n=\sum_{n\ge 0}\frac{1}{(n+1)!}(R_a-L_a)^n=F(\ad a).
$$
If $\|a\|<\pi$, then $\|\ad a\|<2\pi$ hence $\sigma(\ad a)\subset B(0,2\pi)$ so the spectrum of $\ad a$ does not intersect the zero set of $F$. The last assertion is due to the fact that, when $a$ is skew-hermitian, then $e^a$ is a unitary element, hence
$$
\|d\exp_a(b)\|_p\le \int_0^1\|e^{(1-t)a}be^{ta}\|_p\, dt=\|b\|_p.
$$
\qedhere
\end{proof}

The following elementary lemma will simplify the proof of the next theorem.

\begin{lem}\label{fseg}
Let $C,\varepsilon >0$, let $f(-\varepsilon,1+\varepsilon)\to \mathbb R$ be a non constant real analytic function such that $f'(s)^2\le C f''(s)$ for any $s\in [0,1]$. Then $f$ is strictly convex in $(0,1)$.
\end{lem}
\begin{proof}
By the mean value theorem, the condition on $f$ implies that for each pair of roots of $f'$, there is another root of $f'$ in between. Since $f$ is analytic and non-constant, the set of roots of $f'$ is an empty set or has one point $\alpha\in (-\varepsilon,1+\varepsilon)$. If this set of roots does not intersect $(0,1)$, then $f''>0$ there and we are done. We assume then that there exists $\alpha$ in $(0,1)$ such that $f'(\alpha)=0$. Note that  $-f'(x)=f'(\alpha)-f'(x)=\int_x^{\alpha}f''(s)ds> 0$ for any $x\in (-\varepsilon,\alpha]$ and $f'(y)=f'(y)-f_p'(\alpha)=\int_{\alpha}^y f''(s)ds>0$ for any $y\in [\alpha,1+\varepsilon)$, hence $f'$ is strictly negative in $(-\varepsilon,\alpha)$ and strictly positive in $(\alpha,1+\varepsilon)$, so $f$ is strictly convex in each interval. If $f(\alpha)<[f(1)-f(0)]\alpha+f(0)$, we are done. If not, by the mean value theorem there exists $x\in (0,\alpha)$, $y\in (\alpha,1)$ such that
$$
f(1)-f(0)=\frac{f(\alpha)-f(0)}{\alpha}=f'(x)<0
$$
and 
$$
f(1)-f(0)=\frac{f(1)-f(\alpha)}{1-\alpha}=f'(y)>0,
$$
a contradiction.
\end{proof}

\begin{rem}\label{jesiano}
The Hessian of the $p$-norms was studied in \cite{convexg,cocomata}. We recall a few facts we will use in the proof of the next theorem. Let $a,b,c\in \bpe_{ah}$, let $H_a:\bpe_{ah}\to \mathbb R$ stand for the symmetric bilinear form given by
$$
H_a(b,c)=(-1)^{\frac{p}{2}}p \sum_{k=0}^{p-2}\, Tr(a^{p-2-k}b a^k c).
$$
If $Q$ is the quadratic form associated to $H$, then (cf. Lemma 4.1 in \cite{convexg} and equation (3.1) in \cite{cocomata}):
\begin{enumerate}
\item $Q_a([b,a])\le 4 \|a\|_{\infty}^2 Q_a(b)$.
\item $Q_a(b)=p\|ba^{\frac{p}{2}-1}\|_2^2+\frac{p}{2}\sum_{l+m=n-2}\|a^l (ab+ba)a^m\|_2^2$.
\end{enumerate}
In particular $H_a$ is positive definite for any $a\in \bpe_{ah}$.
\end{rem}

Our convexity results follow. If $u\in\upe$, denote by $B_p(u,r)=\{w\in\upe: d_p(u,w)<r\}$.

\medskip

\begin{teo}\label{teoconvexidad1}
Let $p$ be a positive even integer, $u\in \upe$ and  $\beta:[0,1]\to\upe$ a non-constant geodesic contained in the geodesic ball of radius $\frac{\pi}{2}$, namely $\beta\subset
B_p(u,\frac{\pi}{2})$. Assume further that $u$ does not belong to any prologantion of $\beta$. Then 
$$
f_p(s)=d_p(u,\beta(s))^p
$$
is a strictly convex function.
\end{teo}
\begin{proof}
We may assume that $u=1$ since the action of unitary elements is isometric. Let
$v,z\in \bpe_{ah}$ such that $\beta(s)=e^ve^{sz}$. Let
$w_s=log(\beta(s))=log(e^ve^{sz})$, and $\gamma_s(t)=e^{tw_s}$. 

Now $\|w_s\|\le \|w_s\|_p<\pi/2$, so $\gamma_s$ is a short geodesic joining $1$ and $\beta(s)$, of length
$\|w_s\|_p=d_p(1,\beta(s))$. Then $f_p(s)=\|w_s\|_p^p=Tr((-w_s^2)^\frac{p}{2})=(-1)^{\frac{p}{2}}Tr(w_s^p)$, hence 
$$
f'_p(s)=
(-1)^{\frac{p}{2}}p\, Tr(w_s^{p-1} \dot{w_s})=\frac{1}{p-1} H_{w_s}(\dot{w_s},w_s).
$$
For $x,y\in B^p_{ah}$, we have the formula $d\,
\exp_x(y)=\int_0^1 e^{(1-t)x}ye^{tx}\,dt$ from the previous lemma. Since 
$e^{w_s}=e^ve^{sz}$, then $e^{-w_s}\; d\; \exp_{w_s} (\dot{w_s}) =z$, namely
\begin{equation}\label{difexp}
z= \int_0^1 e^{-tw_s} \dot{w_s} e^{tw_s}\;dt.
\end{equation}
Thus $Tr(w_s^{p-1} \dot{w_s} )=\int_0^1 Tr(w_s^{p-1}e^{-tw_s} \dot{w_s} e^{tw_s})\;dt=Tr(zw_s^{p-1})$.
Hence 
$$
f''_p(s)=
(-1)^{\frac{p}{2}}p \sum_{k=0}^{p-2}\, Tr(w_s^{p-2-k}\dot{w_s}w_s^k z)=H_{w_s}(\dot{w_s},z),
$$
and again by equation (\ref{difexp}) above, if we put $\delta_s(t)=e^{-tw_s}\dot{w_s} e^{tw_s}$, then
$$
f_p''(s)=\int_0^1 H_{w_s}(\delta_s(0),\delta_s(t))\,dt.
$$
Suppose that for this value of $s\in[0,1]$, $R_s^2:=Q_{w_s}(\dot{w_s})\ne 0$, where $Q_{w_s}$ is the quadratic form associated to $H_{w_s}$. If $K_s\subset \bpe_{ah}$ is the null space of $H_{w_s}$, consider the quotient space $\bpe_{ah}/K_s$ equipped with the inner product $H_{w_s}(\cdot,\cdot)$. An elementary computation shows that $\delta_s(t)$ lives in a sphere of radius $R_s$ of this pre-Hilbert space, hence
$$
H_w(\delta_s(0),\delta_s(t))=R_s^2 \cos(\alpha_s(t)),
$$
where $\alpha_s(t)$ is the angle subtended by $\delta_s(0)$ and $\delta_s(t)$. Then, reasoning in the sphere
$$
R_s\alpha_s(t)\le L_0^t(\delta_s)=\int_0^t Q_{w_s}^{\frac12}(e^{-tw_s}[w_s,\dot{w_s}]e^{t w_s})\,dt =\int_0^t Q_{w_s}^{\frac12}([w_s,\dot{w_s}])\,dt=t\,Q_{w_s}^{\frac12}([w_s,\dot{w_s}]).
$$
By property $1.$ of above remark,
$$
R_s\alpha_s(t)\le t\, 2\|w_s\|_{\infty} R_s \le 2t \|w_s\|_p R_s<R_s \pi
$$
if $\|w_s\|_p<\frac{\pi}{2}$. So
$$
\cos(\alpha_s(t))\ge \cos(2t\|w_s\|_p)
$$
and then integrating with respect to the $t$-variable,
$$
f''_p(s)\ge R_s^2 \frac{\sin(2\|w_s\|_p)}{2\|w_s\|_p}>0
$$
provided $R_s\ne 0$. On the other hand, the Cauchy-Schwarz inequality for $H_{w_s}$ shows that if $R_s=0$, then 
$$
(p-1)f_p'(s)=H_{w_s}(w_s,\dot{w_s})\le Q^{\frac12}_{w_s}(\dot{w_s})Q^{\frac12}_{w_s}(w_s)=0.
$$
Assume that $R_s$ is identically zero, $s\in [0,1]$. Then $f_p$ is constant with $f_p(s)=f_p(0)=\|v\|_p$ for any $s\in [0,1]$. Moreover, by property $2.$ of the remark above, $R_s=0$ implies $w_s^{\frac{p}{2}-1}z=0$ and an elementary computation involving the functional calculus of skew-adjoint operators shows that $w_s z=0$; in particular $vz=0$ which implies $w_s=v+sz$ by the Baker-Campbell-Hausdorff formula. But since the norm of $\bpe$ is strictly convex, $w_s$ cannot have constant norm unless $v$ is a multiple of $z$, and in that case, $u$ and $\beta$ are aligned contradicting the assumption of the theorem. So there is at least one point $s_0\in  [0,1]$ where $R_{s_0}\ne 0$, so $f_p$ is non constant and by Lemma \ref{fseg}, $f_p$ is strictly convex since
$$
(p-1)^2f_p'(s)^2=H_{w_s}^2(w_s,\dot{w_s})\le Q_{w_s}(\dot{w_s})Q_{w_s}(w_s)\le C f''_p(s).
$$
\end{proof}

\begin{rem}
A careful reading of the proof of the above theorem shows that $f_p$ is in fact strictly convex provided that the \textit{uniform} norm $\|w_s\|$ is strictly less than $\pi/2$. 
\end{rem}

\begin{coro}\label{teoconvexidad2}
Let $u_1,u_2,u_3\in \upe$ with $u_2,u_3\in B_p(u_1,\frac{\pi}{4})$, and assume that they are not aligned (i.e. they do not lie in the same geodesic). Let $\gamma(s)$ be the short geodesic joining $u_2$ with $u_3$. Then $dist_p(u_1,\gamma(s))<\frac{\pi}{2}$ for $s\in[0,1]$ and $\frac{\pi}{4}$ is the radius of convexity of the metric balls of $\upe$.
\end{coro}
\begin{proof}
Note that 
\begin{eqnarray}
dist_p(u_1,\gamma(s)) & \le & dist_p(u_1,u_2)+\frac12 dist_p(u_2,u_3)  \nonumber\\
 & \le  & dist_p(u_1,u_2) + \frac12 (dist_p(u_2,u_1)+dist_p(u_3,u_1))
< 2\frac{\pi}{4}=\frac{\pi}{2} \nonumber,
\end{eqnarray}
hence the conclusion follows from the previous theorem.
\end{proof}

\bigskip

\section{Minimality in $\oo$: initial values problem}

For our main result on minimality in $\oo$, we make the  assumption that for some (hence for any)  $x\in\oo$, the group $G_x$  is locally exponential. Namely, there exists a radius $\delta>0$ such that if  $v\in G_x$ with $\|v-1\|_p<\delta$, then there exists an element $z\in \g_x$ such that $v=e^z$.  This is equivalent to the fact that $G_x$ is a (non complemented) Banach-Lie subgroup of $\upe$. This property implies in particular, that $G_x$ is locally geodesically convex: given any pair of elements $v_1,v_2 \in G_x$ with $\|v_1-v_2\|_p<\delta$, then there exists a unique minimal geodesic of $\upe$, which lies inside $G_x$, and joins $v_1$ and $v_2$.

Our argument on minimality in $\oo$ will consist in comparing the lengths of the liftings of curves to the unitary group $\upe$. For the case $p=2$ this technique is based on the following fact: 

\begin{rem}
Let $\gamma(t)$, $t\in[0,1]$ be a smooth curve in $\oo$, with $\gamma(0)=x$,  and let $\Gamma$ be its horizontal lifting. Then
$$
L_2(\Gamma)=L_2(\gamma).
$$
Indeed, recall from (\ref{levantamientohorizontal}) that $\dot{\Gamma}=\kappa_\gamma(\dot{\gamma}) \Gamma$, and also note that by definition of the metric, $\kappa_x:(T\oo)_x\to  \f_{x}\subset \beah$ is isometric. Then 
$$
\|\dot{\gamma}\|_\gamma=\|\kappa_\gamma(\dot{\gamma})\|_2=\|\Gamma^* \dot{\Gamma}\|_2=\|\dot{\Gamma}\|_2,
$$
and the result follows.
\end{rem}

Let us show that for $p>2$ we can still have isometric lifts of curves in ${\cal O}$. 

First note that the general theory ensures the existence of piecewise $C^1$ liftings in $\upe$ of $C^1$ curves in $\oo$, due to the fact that for any fixed $x\in\oo$, the map
$$
\pi_x:\upe\to \oo, \pi_x(u)=u\cdot x,
$$
is a submersion.
 
We need to discuss the projection to closed linear spaces in $\bpe$ and a few technical lemmas first.

\begin{rem}\label{proy}
Let $1<p<\infty$. Then for any convex closed set $S\subset \bpe_{ah}$ there exists a continuous map $Q_S:\bpe_{ah}\to S$ which sends $x\in \bpe_{ah}$ to its best approximant $Q_S(x)\in S$, i.e. 
$$
\|x-Q_S(x)\|_p\le \|x-s\|_p
$$
for any $s\in S$.

The map $Q_S$ is single-valued and continuous,  because $\bpe$ is uniformly convex and uniformly smooth (see for instance \cite{chongli}). Note that 
$$
\|Q(x)\|_p\le \|Q(x)-x\|_p+\|x\|_p\le \|0-x\|_p+\|x\|_p=2\|x\|_p
$$
and also that
$$
\|x-Q_S(x)-s\|_p\ge \|x-Q_S(x)\|_p
$$
for any $s\in S$, hence $Q_S(x-Q_S(x))=0$, namely $Q_S\circ (1-Q_S)=0$. Also, for any positive $\lambda\in\mathbb{R}$,
$$
Q_S(\lambda x)=\lambda Q_S(x).
$$

Let $x\in {\cal O}$, let $G=G_{x}$ be the isotropy group and ${\cal G}_x$ the Lie algebra of $G$ as usual. Let $S=\g_x$ and $Q=Q_{\g_x}$ be the projection to the best approximant in $\g_x$. Let
$$
\g_x^{\perp_p}=Q^{-1}(0)=\{x\in \bpe_{ah}:\|x\|_p\le \|x-y\|_p \; \mbox{ for any } \; y\in \g_x\}.
$$
Then any element $z\in \bpe_{ah}$ can be decomposed as
$$
z=z-Q(z)+Q(z),
$$
where $z-Q(z)\in \g_x^{\perp_p}$ and $Q(z)\in \g_x$. In particular, these facts imply that given $x\in\oo$ and $X\in(T\oo)_x$,  there exists a minimal lifting $z_0\in\bpe_{ah}$ for $x$. Indeed, since 
$$
\pi_x:\upe\to \oo, \ \pi_x(u)=u\cdot x
$$
is a smooth submersion, the differential $(d\pi_x)_1$ is surjective, and thus there exists $z\in\bpe_{ah}$ such that $d(\pi_x)_1(z)=X$. Then a minimal lifting is
$$
z_0=z-Q(z)\in \g_x^{\perp_p}.
$$
Calling $\bar{Q}=1-Q$, we have 
$$
\g_x=\bar{Q}^{-1}(0)=Im(Q),\qquad \g_x^{\perp_p}=Q^{-1}(0)=Im(\bar{Q}),\qquad 
$$
and also
$$
\bar{Q}^2=\bar{Q},\quad Q^2=Q,\quad \bar{Q}\circ Q=Q\circ \bar{Q}=0.
$$

\end{rem}

\begin{lem}\label{perpep}
Let $p$ be an even positive integer. Let $x\in\oo$ and $X\in(T\oo)_x$.  An element $z\in\bpe_{ah}$ with $(d\pi_x)_1(z)=X$ is a minimal lifting for $X$ if and only if $Tr(z^{p-1}y)=0$ for all $y\in\g_x$. For any $X\in(T\oo)_x$ there exists a unique minimal lifting $z\in \g_x^{\perp_p}$ such that $\|z\|_p=\|X\|_x$.
\end{lem}
\begin{proof}
Suppose that $z_0$ is a minimal lifting, and for a fixed $y\in\g_x$, let $f(t)=\|z_0-ty\|_p^p$. Then $f$ is a smooth map with a minimum at $t=0$, i.e. $f'(0)=0$. A straightforward computation shows that $f'(t)=Tr((z_0-ty)^{p-1}y)$, and thus $Tr(z_0^{p-1}y)=0$. Conversely, suppose that $Tr(z_0^{p-1}y)=0$ for all $y\in\g_x$ and suppose that there exists $y_0\in\g_x$ such that $\|z_0-y_0\|_p<\|z_0\|_p$. Then the map $f(t)=\|z_0-ty_0\|_p^p$ would not have a minimum at $t=0$. This is a contradiction, since $f$ is convex and $f'(0)=0$. The  existence of minimal liftings was established in the previous remark: take any $w\in \bpe_{ah}$ such that $(d\pi_x)_1(w)=X$ and then take $z=w-Q_{\g_x}(w)$. If $(d\pi_x)_1(z_1)=(d\pi_x)_1(z_2)=X$ for $z_1,z_2\in \bpe_{ah}$, then $z_1-z_2\in \g_x$; if $z_1$ and $z_2$ are minimal liftings of $X$, then we have $\|z_1\|_p\le \|z_1-(z_1-z_2)\|_p=\|z_2\|_p$ and the reversed inequality also holds, hence $\|z_1\|_p=\|z_2\|_p=\|X\|_x$. To prove uniqueness we may assume then that $\|z_1\|_p=\|z_2\|_p=1$. Consider the smooth convex function $g:\g_x\to \mathbb R_{>0}$ given by
$$
y\mapsto \|z_1-y\|_p^p.
$$
Now $g(0)=\|z_1\|_p^p=1$ is a minimum for $g$, and we are assuming that $g(z_1-z_2)=\|z_2\|_p^p=1$ is another minimum. Hence $g$ must be constant on the straight segment $s(z_1-z_2)\in\g_x$ for any $s\in [0,1]$. In particular (with $s=\frac12$), 
$$
\|\frac12 (z_1+z_2)\|_p^p=\|z_1\|_p^p=\|z_2\|_p^p=1,
$$
which forces $z_1=z_2$, since $\bpe$ is uniformly convex.
\end{proof}

Having established the linear result on minimal liftings, let us prove two technical lemmas in order to extend the isometric lifting property  to smooth curves  $\gamma\subset {\cal O}$.

\begin{lem}
Let $k\ge 1$, $w\in \bpe$ with $\|w\|_p<\frac{\pi}{2}$. Then
$$
T=1+\frac{(\ad w)^2}{4k^2\pi^2}
$$
is invertible in ${\cal B}(\bpe)$ and 
$$
 \|T^{-1}\|\le \left(1-\frac{\|w\|^2}{k^2\pi^2} \right)^{-1}\le \left(1-\frac{\|w\|_p^2}{k^2\pi^2} \right)^{-1}.
$$
\end{lem}
\begin{proof}
Since $\|\ad w\|\le 2\|w\|\le 2\|w\|_p<\pi$, the map $T$ is invertible and its inverse can be computed with the Neumann series.
\end{proof}

\begin{rem}
Consider  $\displaystyle g(r)=\frac{r}{\sin(r)}$  with $g(0)=1$. Then $g:[0,\pi)\to \mathbb R$ is positive and increasing, and from the Weierstrass expansion of $\sin(z)$ we obtain
$$
g(z)=\prod_{k\ge 1} \left(1-\frac{z^2}{k^2\pi^2} \right)^{-1},
$$
for any $z$ such that $|z|<\pi$.
\end{rem}

\begin{prop}\label{seno}
Let $F(z)=\displaystyle\frac{e^z-1}{z}$, $\displaystyle g(r)=\frac{r}{\sin(r)}$. Let $w\in \bpe$ with $\|w\|_p<\frac{\pi}{2}$. Let $t\in [0,1]$. Then
$$
\|F( \ad w)^{-1}\|\le \,g(\|w\|)\le \, g(\|w\|_p).
$$
\end{prop}
\begin{proof}
The Weierstrass expansion of $F(z)=\frac{e^z-1}{z}$ is given by 
$$
F(z)=\prod_{k\ge 1} \left(1+\frac{z^2}{4k^2\pi^2} \right)
$$
where the product converges uniformly on compact sets to $F$. Then $F(\ad w)$ is invertible since $\|\ad w\|<\pi$ and
$$
F(\ad w)^{-1}= \prod_{k\ge 1} 
\left(1+\frac{(\ad(w))^2}{4k^2\pi^2}\right)^{-1}.
$$
Hence
$$
\|F(\ad w)^{-1}\|\le \prod_{k\ge 1} \left(1-\frac{\|w\|^2}{k^2\pi^2}\right)^{-1}=g(\|w\|)\le g(\|w\|_p)
$$ 
by the previous lemma.
\end{proof}

\begin{lem}
Let $1<p<\infty$, $x\in {\cal O}$  and  $Q=Q_{{\cal G}_x}$ be the best approximant projection. Let $\Gamma\subset\upe$ be a piecewise $C^1$ curve parametrized in the  interval $[0,1]$. Then there exists a piecewise $C^1$ curve $z:[0,1]\to {\cal G}_x$ with $z(0)=0$ such that
$$
F( \ad z)\dot{z}=-Q(\Gamma^*\dot{\Gamma}).
$$
If $u_{\Gamma}=e^z\in G_x$, then $u_{\Gamma}:[0,1]\to \bpe$ obeys the differential equation
$$
\dot{u_{\Gamma}}u_{\Gamma}^*=-Q(\Gamma^*\dot{\Gamma}),
$$
and $L_p(u_{\Gamma})\le 2 L_p(\Gamma)$.
\end{lem}
\begin{proof}
Assume first that $\Gamma$ is $C^1$ in the whole $[0,1]$. Let $R_0=\max\limits_{t\in \overline{J}}\|\dot{\Gamma}\|_p$, where $J$ is an open interval containing $[0,1]$ where $\Gamma$ is differentiable. Let $0<R<\frac{\pi}{2}$. Then if $x\in {\cal G}\cap B(0,R)$, the map $F(\ad x)$ is invertible by the previous lemma,  its inverse is analytic and can be written as a power series in $\ad x$, hence 
$$
F(\ad x)^{-1}:{\cal G}\to {\cal G}
$$
because ${\cal G}$ is a Banach-Lie algebra. Moreover, since $g$ is increasing,
$$
\|F(\ad x)^{-1}\|\le g(\|x\|_p)\le g(R).
$$
Let $f:J\times B(0,R)\cap {\cal G}\to {\cal G}$ be given by
$$
f(t,x)=-F(\ad x)^{-1}Q_{\cal G}(\Gamma^*(t)\dot{\Gamma}(t)).
$$
Then $f$ is continuous since $Q$ and $F^{-1}$ are continuous, moreover
$$
\|f(t,x)\|_p\le \|F(\ad x)^{-1}\| \, 2\|\dot{\Gamma}(t)\|_p\le g(R)2R_0=L
$$
by Remark \ref{proy} and the previous lemma. Since $H(\ad x)=F(\ad x)^{-1}$ is analytic in the ball $\|x\|_p< \frac{\pi}{2}$, we have 
$$
\|H(\ad x)-H(\ad y)\|\le C(R) \|\ad x-\ad y\|\le 2C(R)\|x-y\|_p
$$
where $C(R)$ is the bound for $H'$ in $\|z\|_p\le R$. Then
$$
\|f(t,x)-f(t,y)\|_p\le 4C(R)R_0\|x-y\|_p=K\|x-y\|_p.
$$
Then $f$ satisfies a Lipschitz condition, uniformly respect to $t\in J$, hence by Proposition 1.1 of Ch. IV in \cite{lang}, there exists a continuous solution $z_0: (-b,b)\times B(0,R/4)\to {\cal G}\cap B(0,R)$ of the integral equation
$$
z(t)=\int_0^t f(s,z(s))\,ds
$$
with $z_0(0)=0$. Here $b$ is any real number 
$$
0<b<\frac{R}{4LK}=\frac{\sin(R)}{32C(R)R_0^2}
$$
Note that $z_0$ is in fact $C^1$. Differentiating both sides and multiplying by $F(\ad z(t))$ gives the equation stated. We have proved so far that the equation 
$$
F(\ad z)\dot{z}=-Q(\Gamma^*\dot{\Gamma})
$$
has a local solution defined around zero. By a standard argument, it follows that one can find a piecewise $C^1$ solution defined on the whole interval $[0,1]$: let $N\in \mathbb N$ such that $\frac{1}{N}<b$ and let $t_k=\frac{k}{N}$. Then $[t_k,t_{k+1}]$ ($k=0,1,\cdots N$) is a partition of $[0,1]$ such that the integral equation
$$
z(t)=\int_{t_k}^{t_{k+1}} f(s,z(s))\,ds
$$
with the initial conditions $z_0(0)=0$, $z_k(t_k)=z_{k-1}(t_k)$ for $k\ge 1$, has a solution $z_k:[t_k,t_{k+1}]\to \cal G$. Then the curve $z_1 \sharp z_2\sharp\cdots\sharp z_N$ is a piecewise $C^1$ solution of the equation defined in the whole $[0,1]$. If $\Gamma$ is piecewise $C^1$ instead of $C^1$, one might replace the argument above for a similar argument in each of the intervals where $\Gamma$ is $C^1$, and use the continuity of $\Gamma$ to state the boundary conditions for $z$.

If $u_{\Gamma}(t)=e^{z(t)}$, then 
$$
\dot{u}_{\Gamma}(t)=d\exp_{z(t)}(\dot{z}(t))=\int_0^1 e^{sz(t)}\dot{z}(t) e^{-sz(t)}ds\, u_{\Gamma}(t)=F(\ad z(t))\dot{z}(t) u_{\Gamma}(t)
$$
by Lemma \ref{expo}. Then $\dot{u}=F(\ad z)\dot{z}\,u$, and hence $\dot{u}u^*=-Q(\Gamma^*\dot{\Gamma})$. Thus
$$
\|\dot{u}\|_p=\|Q(\Gamma^*\dot{\Gamma})\|_p\le 2 \|\Gamma^*\dot{\Gamma}\|_p=2\| \dot{\Gamma}\|_p,
$$
and therefore $L_p(u)\le 2L_p(\Gamma)$.
\end{proof}

\begin{prop}\label{lift}
Let $x_0\in \oo$,  $\gamma=\Gamma\cdot x_0\subset \oo$ a $C^1$ curve defined in an interval containing $[0,1]$. Then $\gamma$ admits a piecewise $C^1$ lift $\beta\subset\upe$ (that is $\beta\cdot x_0=\gamma$) such that $L(\gamma)=L_p(\beta)\le L_p(\Gamma)$. We shall call $\beta$ an \textbf{isometric lift} of $\gamma$.
\end{prop}
\begin{proof}
Let $u=u_{\Gamma}=e^z$ be the curve of the previous lemma. Then $\beta=\Gamma \, u_{\gamma}$ is a lift of $\gamma$ because $u\in G$. Moreover,
\begin{eqnarray}
\|\dot\beta\|_p & = & \|\dot{\Gamma}u+\Gamma \dot{u}\|_p=
\| \Gamma^*\dot{\Gamma}+ \dot{u}u^*\|_p  =  \|\Gamma^*\dot{\Gamma}-Q(\Gamma^*\dot{\Gamma})\|_p\nonumber\\
& =&\min\limits_{y\in{\cal G}} \|\Gamma^*\dot{\Gamma}-y\|_p\le \|\Gamma^*\dot{\Gamma}\|_p=\|\dot{\Gamma}\|_p,\nonumber
\end{eqnarray}
hence $L(\gamma)=L_p(\beta)\le L_p(\Gamma)$.
\end{proof}

\begin{teo}
Let $p$ be a positive even integer, $x\in\oo$, $X\in (T\oo)_x$ and $z_0\in\bpe_{ah}$  a minimal lifting for $X$. Then the curve
$$
\delta(t)=e^{t z_0}\cdot x,
$$
which verifies $\delta(0)=x$ and $\dot{\delta}(0)=X$, has minimal length  $\|z_0\|_p$ in the interval $[0,1]$ if $\|z_0\|_p<\pi/4$. Moreover, the curve $\delta$ is unique with this property, in the sense that if $\gamma\subset {\cal O}$ is another curve joining $x$ to $e^{z_0}\cdot x$ of length $\|z_0\|_p$, then $\gamma(t)=e^{tz_0}\cdot x$.
\end{teo}
\begin{proof}
Let $\gamma$ be a smooth curve in $\oo$ with $\gamma(0)=x$ and $\gamma(1)=e^{z_0}\cdot x$.
Denote by $\beta$ as above an isometric lift of $\gamma$. Note that the curve $\epsilon(t)=e^{tz_0}$ is an isometric lift for $\delta$. 
Then it suffices to compare $\beta$ and $\epsilon$ (note that both curves start at $1$). There exists in $\upe$ a minimal curve $\alpha(t)=e^{tz}$ with $e^z=\beta(1)$, with $L_p(\alpha)\le L_p(\beta)$. We claim that $L_p(\epsilon)\le L_p(\alpha)$, a fact which ends the proof.
If $\|z\|_p=L_p(\alpha)>\pi/4$, this fact is clear. Suppose that $\|z\|_p\le \pi/4$. Let $\nu(t)=e^{z_0}e^{ty}$ be the minimal geodesic of $\upe$, lying inside $e^{z_0} G_x$ (i.e. $y\in\g_x$), connecting $e^{z_0}$ to $e^{z}$. Then by  Theorem \ref{teoconvexidad1}, the map $f_p(s)=d_p^p(1,\nu(s))$ is convex. We claim that $f_p'(0)=0$, and thus 
$$
L_p(\epsilon)^p=d_p(1,\nu(0))^p=f_p(0)\le f_p(1)=d_p(1,\nu(1))^p=L_p(\alpha)^p.
$$
As in the proof of \ref{teoconvexidad1}, $f_p'(0)=(-1)^{p/2} Tr(z_0^{p-1}y)$, which vanishes by Lemma \ref{perpep}, because $z_0$ is a minimal lift. If $L(\gamma)=\|z_0\|_p$ (i.e. if $\gamma$ is also short), then
$$
f_p(1)=\|z\|_p^p\le L_p(\beta)^p=L(\gamma)^p=\|z_0\|_p^p=f_p(0)
$$
and then $z=z_1$ because $f_p$ is strictly convex. In particular $\beta(1)=e^{z_0}$ and $L_p(\beta)=L_p(\epsilon)=\|z_0\|_p$. Since $\|z_0\|\le\|z_0\|_p<\pi/2$, the curve $\epsilon$ is the unique short geodesic joining $1$ to $e^{z_0}$ in $\upe$,  and then $\beta=\epsilon$.
\end{proof}

\section{Completeness of the metric spaces ${\cal O}$}

We prove that the space ${\cal O}$ is a  complete metric space with the rectifiable metric. Let us prove first an inequality in $\upe$ relating the distance among two geodesics with the distance of the endpoints. Throughout we assume that $p$ is a positive even integer.

\begin{teo}
Let $g(r)=\displaystyle\frac{r}{\sin(r)}$. Let $u,v,w\in \upe$ with $v,w\in B_p(u,r_0)$ and $r_0\in [0,\frac{\pi}{4}]$. Let $\gamma$ be the short geodesic joining $v$ to $w$. Let $\alpha$ (resp. $\beta$) be the short geodesic joining $u$ to $v$ (resp. $u$ to $w$). Let $\gamma_t$ be the short geodesic joining $\alpha(t)$ with $\beta(t)$. Then
$$
L_p(\gamma_t)\le t\, g(r_0)\,  L_p(\gamma)\le \frac{\pi\,t}{2\sqrt{2}} \,L_p(\gamma)
$$
for any $t\in [0,1]$.
\end{teo}
\begin{proof}
We may suppose $u=1$ without loss of generality.  Let $\gamma^t=e^{t log(\gamma)}$. 
Since $\gamma^t(0)=\alpha(t)$ and $\gamma^t(1)=\beta(t)$, and  $\gamma_t$ is a short geodesic joining these same endpoints, one has the inequality
$$
L_p(\gamma_t)\le L_p(\gamma^t).
$$
Let us use the dot to denote the derivative with respect to the $s$ variable. Then
$$
\dot{\gamma^t}=d\exp_{t log(\gamma)}(t\, d log_{\gamma}(\dot{\gamma}))=t\, d\exp_{t\ln(\gamma)}(d\exp_{\gamma}^{-1}(\dot{\gamma})),
$$
hence
$$
\|\dot{\gamma^t}\|_p\le\,t\, \|d\exp_{\gamma}^{-1}(\dot{\gamma})\|_p
$$
by Lemma \ref{expo}, since  the differential of the exponential map is a contraction. By the same lemma,
\begin{equation}\label{ecua}
\|\dot{\gamma^t}\|_p\le t\,\|F(\ad(log(\gamma))^{-1} \dot{\gamma}\|_p\le t\, g(\|log(\gamma)\|)\,\|\dot{\gamma}\|_p
\end{equation}
where the last inequality is due to Proposition \ref{seno}. Now by Corollary \ref{teoconvexidad2}, $\|log(\gamma)\|\le \|log(\gamma)\|_p<r_0<\frac{\pi}{4}$, and since $g$ is increasing in $[0,\pi)$, the term $g(\|log(\gamma)\|)$ is bounded by $g(r_0)$, which in turn is bounded by $g(\frac{\pi}{4})=\frac{\pi}{2\sqrt{2}}$. Integrating (\ref{ecua}) with respect to the variable $s$ in $[0,1]$ gives the inequalities for the $p$-lengths.
\end{proof}

\begin{coro}\label{wsp}
Let $u_1,u_2,u_3\in \upe$ such that $d_p(u_i,u_j)<r_0\le \frac{\pi}{4}$, $u_2=u_1e^x$, $u_3=u_1e^y$. Then
$$
|d_p(u_1,u_2)-d_p(u_1,u_3)|\le\|x-y\|_p\le g(r_0) \, d_p(u_2,u_3).
$$
\end{coro}
\begin{proof}
The first inequality is just the reversed triangle inequality, since $\|x\|_p=d_p(u_1,u_2)$ and $\|y\|_p=d_p(u_1,u_3)$. By the invariance of the metric under left action of the unitary group, we may assume that $u_1=1$. Then for each $t\in [0,1]$, (in the notation of the previous result)
$$
d_p(e^{tx},e^{ty})=\|log(e^{tx}e^{-ty})\|_p=L_p(\gamma_t),$$
which is less or equal than $g(r_0)\,t\, d_p(u_2,u_3)$ by the same proposition. Then
$$
\|\frac{1}{t}log(e^{tx}e^{-ty})\|_p \le g(r_0)\, \, d_p(u_2,u_3),
$$
and taking the limit $t\to 0^+$ gives the result.
\end{proof}

\begin{rem}
Recall Clarkson's inequalities \cite{simon} for the $\bpe$ spaces, $p\in [2,+\infty)$,
$$
2\|x\|_p^p+2\|y\|_p^p\le \|x-y\|_p^p+\|x+y\|_p^p,
$$
for any $x,y\in \bpe$.
\end{rem}

\begin{teo}(Weak semi-parallelogram law)\label{semiparalelogramo}
Let $\gamma$ be a short geodesic in $\upe$ and $u\in \upe$ such that $d_p(u,\gamma)<r_0\le\frac{\pi}{4}$. Then 
$$
\frac{1}{2}g(r_0)\left[d_p^p(u,\gamma(0))+d_p^p(u,\gamma(1))\right]-d_p^p(u,\gamma(1/2))\ge \frac{1}{2^p} L_p(\gamma)^p.
$$
\end{teo}
\begin{proof}
We may assume that $\gamma(1/2)=1$. Then $\gamma(0)=e^x$, $\gamma(1)=e^{-x}$ and $u=e^y$ with $x,y\in\bpe_{ah}$. Then, by Clarkson's inequality,
\begin{eqnarray}
\frac{1}{2^p} L_p(\gamma)^p & = &\|x\|_p^p\le \frac12 \left[ \|x+y\|_p^p+\|x-y\|_p^p \right]- \|y\|_p^p\nonumber\\
 &=&\frac12 \left[ \|x+y\|_p^p+\|x-y\|_p^p \right]- d_p^p(u,\gamma(1/2)).\nonumber
\end{eqnarray}
Now apply Corollary \ref{wsp}.
\end{proof}

Let us finish this section by proving completeness of the geodesic distance.

\begin{teo}
The metric space $(\oo,d)$ is complete.
\end{teo}
\begin{proof}
Let $\{b_n\}_{n\ge 1}$ be a Cauchy sequence in $\oo$, and fix $\pi/4\ge \varepsilon>0$.  Then there exists $n_0$ such that $d(b_n,b_m)<\varepsilon$ if $n,m\ge n_0$. Consider the (submersion) map 
$$
\pi=\pi_{b_{n_0}}: \upe \to \oo, \ \ \pi(u)=u\cdot b_{n_0}.
$$
For $n,m\ge n_0$, let $\gamma_{n,m}$ be a smooth curve in $\oo$ joining $b_{n}$ with $b_m$ at (respectively) $t=0$ and $t=1$, such that
$$
d(b_n,b_m)^p\le L(\gamma_{n,m})^p <d(b_n,b_m)^p+\varepsilon. 
$$
Then by Proposition \ref{lift} the curve $\gamma_{n_0,m}$ lifts, via $\pi$, to a curve $\mu_m$ of $\upe$ with $\mu_m(0)=1$, 
$$
\pi(\mu_{m}(t))=\gamma_{n_0,m}(t), \ \ t\in[0,1],
$$
such that $L_p(\mu_m)=L(\gamma_{n_0,m})$. 
Denote by $u_m=\mu_m(1)$. Then
$$
\varepsilon + d(b_{n_0},b_{m})^p>L(\gamma_{n_0,m})^p=L_p(\mu_m)^p\ge d_p(1,u_m)^p.
$$ 
For each $n,m\ge n_0$, let $v_n,z_{n,m}\in\bpe_{ah}$ be such that $\omega_{n,m}(t)=e^{v_n}e^{t z_{n,m}}$ is the unique minimal geodesic in $\upe$ which joins $u_n$ and $u_m$ at $t=0$ and $t=1$.  Then 
$$
d_p(1,\omega_{n,m}(t))<\pi/4
$$
if $\varepsilon$ is small enough. Hence by Theorem \ref{semiparalelogramo}, if $n,m\ge n_0$ then
\begin{eqnarray}
2\varepsilon &> & \varepsilon +\frac12 d(b_{n_0},b_n)^p+\frac12 d(b_{n_0},b_m)^p\ge \frac12 d_p(1,u_n)^p+\frac12 d_p(1,u_m)^p \nonumber\\
&  \ge &\frac{1}{g(\pi/4)2^{p-1}}(d_p(1,\omega_{n,m}(1/2))^p+d_p(u_n,u_m))^p\ge \frac{1}{g(\pi/4)2^{p-1}}d_p(u_n,u_m)^p\nonumber.
\end{eqnarray}
It follows that $\{u_n\}_{n\ge 1}$ is a Cauchy sequence in $\upe$, which is complete. Therefore the sequence $b_n=\pi(u_n)$ is convergent in $\oo$.
\end{proof}

\section{Submanifold  structure of $\oA$}

In this section we consider  the case $\oo=\oA=\{uAu^*:u\in\udos\}$, for a bounded self-adjoint operator $A$, and we study its local structure as a subset of ${\cal B}(\h)$. An elementary computation shows that all elements in $\oA$ are of the form $A+k$ with $k\in\be_h$. If $A$ itself lies in $\be$, then $\oA\subset \be_h$. Otherwise, $\oA\subset A+\be$, which can be regarded as an affine Hilbert space. In either case, a natural question is whether the manifold $\oA$ is a differentiable submanifold of the ambient Hilbert space. This is the purpose of this section. We show that the orbits $\oA$ are not, in general, differentiable submanifolds of $A+\be$. We show that $\oA\subset\be_h$ is a differentiable submanifold if and only if the spectrum of $A$ is finite. 

The obstruction for $\oA$ to be a submanifold is that its tangent spaces may not be closed in $\be$.  The tangent space of $\oA$ at $A$ (i.e. the derivatives at $A$ of smooth curves in $\be$, lying inside $\oA$) is apparently given by
$$
(T\oA)_A=\{xA-Ax: x\in\beah\}.
$$
D. Herrero, D. Voiculescu, C. Apostol and L. Fialkow, among others,  established several important results on the closedness of commutators (see the books  \cite{herrero2,herrero1} and the references therein for a complete  review on the subject). In particular, L. Fialkow \cite{fialkow} addressed the problem of the spectral characterization of Rosenblum's operators restricted to the Schatten ideals. Let us cite Fialkow's result:
Denote by $\tau_{AB}$ the operator $\tau_{AB}(x)=Ax-xB$. Let ${\cal J}$ be any Schatten ideal.
\begin{teo}\label{teofialkow} {\rm (Fialkow \cite{fialkow})}
The following are equivalent
\begin{enumerate}
\item
$\tau_{AB}:\B(\h)\to \B(\h)$ is bounded below.
\item
$\tau_{AB}:{\cal J}\to {\cal J}$ is bounded below for some ${\cal J}$.
\item
$\tau_{AB}:{\cal J}\to {\cal J}$ is bounded below for any ${\cal J}$.
\item
$\sigma_l(A)\cap \sigma_r(B)=\emptyset$.
\end{enumerate}
\end{teo}
Here $\sigma_l(A)$ (resp. $\sigma_r(B)$) denotes the left (resp. right) spectrum of $A$ (resp. $B$).

Recall the map
$$
\pi_A:\udos\to \oA , \ \ \pi_A(u)=uAu^*
$$
and its differential at the identity
$$
\delta_A: \beah \to (T\oA)_A, \ \ \delta_A(x)=xA-Ax.
$$
The Banach-Lie algebra $\beah$ can be decomposed
$$
\beah={\cal G}\oplus {\cal F},
$$
for $\g=\{x\in\beah: xA=Ax\}$ and $\f=\g^\perp$.
Let  
$$
P_A:\beah\to {\cal F}\subset \beah
$$ 
be the orthogonal projection. 
Note that since $\ker \delta_A={\cal F}$, then
$$
\delta_A|_{{\cal F}}:{\cal F} \to (T\oA)_A
$$
is a linear bijection. If $(T\oA)_A\subset \be_h$ were closed, then $\delta_A|_{{\cal F}}$ would be an isomorphism between Banach spaces, and therefore there would exist a constant $C_A$ such that
\begin{equation}\label{acotaciondelta}
\|xA-Ax\|_2\ge C_A \|x-P_A(x)\|_2.
\end{equation}

\begin{teo}\label{partecontinuapartediscreta}
Let $A\in\B(\h)$ self-adjoint. Then $(T\oA)_A\subset \be_h$ is closed if and only if the spectrum of $A$ is finite.
\end{teo}
\begin{proof}
First note that if $(T\oA)_A\subset \be_h$ is closed, then $\{xA-Ax: x\in \be\}\subset \be$ is also closed. Indeed, since $A$ is self-adjoint,  the derivation $\delta_A$, which is clearly defined on $\be$, maps $\beah$ into $\be_h$, and $\be_h$ into $\beah$. Therefore if $\delta_A(x_n)\to y$ in $\be$, and one decomposes $x_n=x_n^{ah}+x_n^h$ in its hermitian and skew-hermitian parts, then both $\delta_A(x_n^{ah})\in\be_h$ and $\delta_A(x_n^h)\in\beah$ are convergent. The hypothesis that $\delta_A(\beah)=(T\oA)_A$ is closed clearly implies that also $\delta_A(\be_h)$ is closed, and our claim follows.
We may decompose $\h=\h_{pp}\oplus\h_c$ in two orthogonal subspaces which reduce $A$, such that $A_{pp}=A|_{\h_{pp}}:\h_{pp}\to \h_{pp}$ has a dense subset of eigenvalues in its spectrum, and $A_c=A|_{\h_c}:\h_c\to \h_c$ has no eigenvalues. Since this decomposition reduces $A$, then clearly $\delta_{A_{pp}}({\cal B}_2(\h_{pp}))$ and $\delta_{A_c}({\cal B}_2(\h_c))$ are closed, by a similar argument as above. 

Let us reason first with  $A_c$, usually called the continuous spectrum part. Note that $\delta_{A_c}$ is injective. If $x\in{\cal B}_2(\h_c)$ satisfies $\delta_{A_c}(x)=0$ then $x$ commutes with $A_c$, and thus the real and imaginary parts of $x$ commute with $A_c$. This means that there is a non zero compact self-adjoint  operator $y$ which commutes with $A_c$. This is clearly not possible: let $p$ be any spectral (finite rank) projection of $y$, then $p$ commutes with $A_c$, and therefore $pA_cp$, being a finite rank self-adjoint operator, would have an eigenvalue, and therefore $A_c$ would have an eigenvalue. It follows that $\delta_{A_C}$ is bounded below, and therefore by Fialkow's theorem \ref{teofialkow}, $\sigma_l(A_c)\cap \sigma_r(A_c)=\emptyset$, which is impossible because for self-adjoint operators, the left and right spectra coincide. 
 
Thus $\h_c=\{0\}$, and the spectrum of $A=A_{pp}+0_{\h_c}$ has a dense subset of eigenvalues. Suppose that there are infinitely many different non zero eigenvalues $\{\lambda_n: n\ge 1\}$, ordered such that $|\lambda_1|\ge |\lambda_2| \ge \dots $.
Let $\{e_n:n\ge 1\}$ be an orthonormal set in $\h$, consisting of the corresponding eigenvectors of $A$. These vectors span a subspace $\h_0$ which reduces $A$. If we denote by $A_0=A|_{\h_0}$, it is clear again that $\delta_{A_0}({\cal B}(\h_0))$ is closed. Thus we may suppose $\h=\h_0$.  We shall write operators in $\h$ as matrices with respect to this basis. 
Let us show  that $\delta_A(\h)$ is not closed. With these reductions, it is clear that ${\cal F}$ consists of diagonal matrices. Therefore  $P_A(x)$ consists of leaving the main diagonal of $x$ fixed, and replacing all non diagonal entries of $x$ with zeros. For each $n\ge 1$, consider the $n\times n$ matrix $b_n$ with $1/n$ in all entries, and $x_n$ the operator on $\h$ with matrix $ib_n$ in the main $n\times n$ corner block, and zeros elsewhere. Note that $x_n$ is  $i$ times a rank one projection, and therefore its $2$-norm is $1$. Also note that $\|P_A(x_n)\|_2=1/\sqrt{n}\to 0$. Therefore 
$$
\|x_n-P_A(x_n)\|_2\to 0.
$$
A straightforward matrix computation shows that $x_nA-Ax_n$ is zero but on the main $n\times n$ corner block, where it has the matrix with  $\frac{1}{n}(a_j-a_i)$ at the $i,j$-entry. Therefore
$$
\|x_nA-Ax_n\|_2^2 =\frac{1}{n^2}\sum_{i,j=1}^n a_j^2+a_i^2-2a_ja_i=\frac{2}{n^2}\{n\sum_{k=1}^na_k^2-\sum_{k=1}^na_k^2\}\le \frac{2}{n}\sum_{k=1}^na_k^2\le \frac{2}{n}\|A\|_2^2,
$$
and thus $\|x_nA-Ax_n\|_2\to 0$. It follows that $(T\oA)_A$ is not closed.

If the spectrum of $A$ is finite, then $A=\sum_{i=1}^n \lambda_i p_i$, for pairwise orthogonal self-adjoint projections $p_i$ which sum $1$. One can write operators in $\be$ as $n\times n$ matrices in terms of the decomposition $\h=\sum_{i=1}^n R(p_i)$. A straightforward computation shows that if $x\in\beah$ with matrix $(x_{i,j})$, then $\delta_A(x)$ is, in matrix form
$$
\delta_A(x)= \left( \begin{array}{lllll}
0 & (\lambda_2-\lambda_1)x_{1,2} & (\lambda_3-\lambda_1)x_{1,3} & \dots & (\lambda_n-\lambda_1)x_{1,n} \\
(\lambda_1-\lambda_2)x_{2,1} & 0 & (\lambda_3-\lambda_2)x_{2,3} & \dots & (\lambda_n-\lambda_2)x_{2,n} \\
\dots & \dots & \dots & \dots & \dots \\
(\lambda_1-\lambda_n)x_{n,1} & (\lambda_2-\lambda_n)x_{n,2} & (\lambda_2-\lambda_1)x_{1,2} & \dots & 0
\end{array} \right).
$$
Since $\lambda_i\ne\lambda_j$ if $i\ne j$, it follows that $\{xA-Ax: x\in \beah\}$ consists of operators in $\be_h$ whose $n\times n$ matrices have zeros on the diagonal, i.e.
$$
\{xA-Ax: x\in \beah\}=\{z\in\be_h: p_i zp_i=0 , \ i=1,\dots , n\},
$$
which is clearly closed in $\be_h$.
\end{proof}
\begin{rem}
If the spectrum of $A$ is finite, the optimal  constant $C_A$ can be computed. If $A=\sum_{i=1}^n p_i$ as above, the set $\{x\in \be: xA=Ax\}$ consists of   block diagonal  matrices. Thus
$$
P_A(x)=\sum_{i=1}^n p_i x p_i .
$$
Using the matrix form of $\delta_A(x)$, 
$$
\|\delta_A(x)\|_2^2=\sum_{i\ne j}\|(\lambda_j-\lambda_i)x_{i,j}\|_2\ge \inf_{i\ne j}|\lambda_j-\lambda_i|^2\sum_{i\ne j}\|x_{i,j}\|_2^2=\inf_{i\ne j}|\lambda_j-\lambda_i|^2 \|x-P_A(x)\|_2^2.
$$
Thus $C_A=\inf_{i\ne j}|\lambda_j-\lambda_i|^2$.
\end{rem}

The finite spectrum situation contains interesting cases. For instance, if $A=P$ is a projection with infinite rank and co-rank, the orbit $\oo$ equals the connected component of the restricted Hilbert-Schmidt Grassmannian  corresponding to the polarization $\h=R(P)\oplus R(P)^\perp$ (see \cite{belrattum,segalwilson}) with virtual dimension $0$ (i.e. the component containing $P$). From the above proposition it is clear that the finite spectrum condition is necessary for $\oA$ to be a submanifold of $A+\beah$ (or a differentiable manifold with the $2$-norm topology). In the rest of this section we shall prove that it is also sufficient. 

To establish the equivalence between the existence of the submanifold structure for $\oA\subset A+\beah$ and the finite spectrum condition, the following general result on homogeneous spaces is useful. A proof can be found in \cite{rae}.
\begin{lem}\label{lemaraeburn}
Let $G$ be a Banach-Lie group acting smoothly on a Banach space $X$. For a fixed
$x\in X$, denote by $\pi_{x}:G\to X$ the smooth map $\pi_{x}(g)=g\cdot
x$. Suppose that
\begin{enumerate}
\item
$\pi_{x}$ is an open mapping, when regarded as a map from $G$ onto the orbit
$\{g\cdot x: g\in G\}$ of $x$ (with the relative topology of $X$).
\item
The differential $d(\pi_{x})_1:(TG)_1\to X$ splits: its kernel and range are
closed complemented subspaces.
\end{enumerate}
Then the orbit $\{g\cdot x: g\in G\}$ is a smooth submanifold of  $X$, and the
map
$\pi_{x}:G\to \{g\cdot x: g\in G\}$ is a smooth submersion.
\end{lem}

\begin{teo}
$\oA\subset A+\beah$ is a differentiable submanifold if and only if the spectrum of $A$ is finite.
\end{teo}
\begin{proof}
The necessary part is clear. Suppose that the spectrum of $A$ is finite, $A=\sum_{i=1}^n \lambda_i p_i$. We shall use Lemma \ref{lemaraeburn} above. Note that in our case $G=\udos$, $d(\pi_{x})_1=\delta_A$. Its kernel is complemented, its range is complemented by the previous theorem. Therefore it remains to prove that $\pi_A:\udos\to \oA$ is open, or equivalently, that it has a local continuous cross section defined on a neighborhood of $A\in\oA$. Since the range of $\delta_A$ is closed, there exists a constant $C_A$ as in (\ref{acotaciondelta}): $\|xA-Ax\|_2\ge C_A\|x-P_A(x)\|_A$, for $x\in\beah$. Note that $P_A$ can be extended to a $\| \ \|$-contractive idempotent map, which we shall still call $P_A$,
$$
P_A:\B(\h)\to \{x\in \B(\h): xA=Ax\} \subset \B(\h), \ \  P_A(x)=\sum_{i=1}^n p_i x p_i.
$$
Clearly, $P_A|_{\be}$  is the $Tr$-orthogonal projection onto the closed subspace $\{x\in \be: xA=Ax\}$. Also it is clear that the inequality (\ref{acotaciondelta}) is still valid for $x\in\be$. Moreover,  $P_A$  has the following modular property: if $y,z\in \{x\in \B(\h): xA=Ax\}$, then $P_A(yxz)=yP_A(x)z$. Consider the open ball $\B=\{b\in \oA: \|b-A\|_2<C_A\}$. We define the following map in ${\cal B}$:
$$
\sigma: {\cal B}\to \udos , \ \ \sigma(b)=u \Omega(P_A(u^*)), \hbox{ if } b=uAu^*,
$$
where $\Omega$ is the unitary part in the polar decomposition of an invertible operator in $Gl_2(\h)$ ($g=\Omega(g)|g|$). Several facts involving the well definition of $\sigma$ need to be checked. First note that $P_A(u)$ lies in $Gl_2(\h)$: since $b=uAu^*\in{\cal B}$, one has that 
$$
C_A\|u-P_A(u)\|_2\le \|uA-Au\|_2=\|uAu^*-A\|_2<C_A,
$$
i.e. $\|1-P_A(u)u^*\|=\|u-P_A(u)\|\le \|u-P_A(u)\|_2<1$, and thus $P_A(u)$ is invertible. Moreover, $P_A(u)-1=P_A(u-1)\in\be$, and therefore $P_A(u)\in GL_2(\h)$, and thus $\Omega(E(u)^*)\in \udos$. Next note that it does not depend on the unitary $u$ performing $b=uAu^*$: if also $b=u'Au'^*$, then $u'=uv$ for $vA=Av$, and thus 
$$
u'\Omega(P_A(u'^*))=uv \Omega(P_A(v^*u^*))=uv \Omega(v^*P_A(u^*))=u\Omega(P_A(u^*)).
$$
Let us prove that $\sigma$ is continuous. It suffices to show that it is continuous at $A$. Suppose that $u_nAu_n^*\to A$. Then  as above, $\|u_nA-Au_n\|_2\to 0$, and therefore $\|u_n-P_A(u_n)\|_2 \to 0$, or equivalently,
$$
\|1-u_n P_A(u_n^*)\|_2=\|1-P_A(u_n)u_n^*\|_2=\|u_n-P_A(u_n)\|_2\to 0.
$$
Therefore (since $\Omega$ is continuous), $\sigma(u_nAu_n^*)=u_n\Omega(P_A(u_n^*))=\Omega(u_n P_A(u_n^*))\to 1$.
Finally, $\sigma$ is a cross section: if $b=uAu^*$,
$$
\sigma(b)A\sigma(b)^*=u\Omega(P_A(u^*)) A \Omega(P_A(u^*))^* u^*=uAu^*,
$$
because the fact that $P_A(u^*)$ commutes with $A$ implies that also $\Omega(P_A(u^*))$ commutes with $A$.
\end{proof}

We finish this section by returning to the case of an arbitrary self-adjoint operator $A$. We shall prove that  the projection
$P_A$ verifies that $\|P_A(x)\|\le \|x\|$.
\begin{prop}
The projection $P_A$ is $\| \ \|$-contractive.
\end{prop}
\begin{proof}
We shall prove this result by giving an alternate construction of $P_A$. Let $\Pi$ be a finite partition of the spectrum of $A$ by Borel sets $\{\Delta_1,\dots , \Delta_{n(\Pi)}\}$. Denote by $p_i$ the spectral projection of $A$ corresponding to the set $\Delta_i$, and by
$$
E_\Pi(x)=\sum_{i=1}^{n(\Pi)}p_i x p_i, \ \ x\in\be.
$$
Consider the partial order $\ge$ on finite partitions given by refinement. Then $\{E_\Pi\}$ is a net of contractions acting in the Hilbert space $\be$. Therefore it has a weak operator convergent subnet, which for simplicity we shall denote again by $\{E_\Pi\}$. Therefore there exists a contraction $F$ acting on $\be$ such that
$$
Tr(y^*E_\Pi(x))\to Tr(y^*F(x)) , \ \ \hbox{ for all } x,y\in\be.
$$
We claim that $F$ is the orthogonal projection onto $\g_p$, i.e. $F=P_A$. 
First note that if $x\in \beah$ commutes with $A$, then it commutes with its spectral projections and therefore 
$$
E_\Pi(x)=\sum_{i=1}^{n(\Pi)}p_i x p_i=x\sum_{i=1}^{n(\Pi)}p_i=x,
$$
therefore $F(x)=x$. Let $p$ be a spectral projection of $A$. Since the index set $\{\Pi\}$ of the convergent net $\{E_\Pi\}$ is co-final, there exists a partition $\Pi_0$ (an index of the net) which is finer than $\{p,1-p\}$. Therefore for any  partition $\Pi=\{p_1,\dots,p_{n(\Pi)}\}\ge \Pi_0$, the projections $p_i$ are either sub-projections of $p$ or $1-p$, and thus $pp_i=p_ip$ equals $p_i$ or $0$. Then 
$$
p E_\Pi(x)=\sum_{pp_i\ne 0}p_ixp_i=E_\Pi(x) p.
$$
It follows that for any $x$, and $\Pi\ge \Pi_0$, $A$ commutes with $E_\Pi(x)$. Then $A$ commutes with $F(x)$. It follows that $F$ is an idempotent operator acting in $\beah$, whose range is $\g_{A}$. Apparently, all the operators $E_\Pi$ are symmetric with respect to the trace inner product, therefore $F$ is symmetric. Then $F$ is the $Tr$-orthogonal projection onto $\g_{A}$, i.e. $F=P_A$. This description of $P_A$ allows us to prove that it is also $\| \ \|$-contractive. 
Indeed, note that for any fixed $x$, the net of operators $\{E_\Pi(x)\}$ converges to $P_A(x)$ in the weak operator topology: if $\xi,\eta\in\h$, denote by $\xi\otimes \eta$ the rank one operator given by $\xi\otimes \eta (\alpha)=<\alpha,\eta>\xi$,
\begin{eqnarray}
<E_\Pi(x)\xi,\eta> & =& Tr((E_\Pi(x)\xi)\otimes \eta)=Tr(E_\Pi(x) \xi\otimes \eta)\nonumber\\
 &=&-<\xi\otimes\eta, E_\pi(x)>\to -<\xi\otimes\eta, P_A(x)>=<P_A(x)\xi,\eta>.\nonumber
\end{eqnarray}
On the other hand, the operators $E_\Pi(x)$ clearly verify $\|E_\Pi(x)\|\le \|x\|$. Then $\|P_A(x)\|\le \|x\|$.
\end{proof}

\section{Open problem: geodesics joining given endpoints}
In this section we consider the problem of finding a minimal curve in $\oo$ joining two given endpoints. First let us remark that the answer is positive, at least locally, for the case $p=2$. In this case $\oo$ is a Riemann-Hilbert manifold, and therefore there exists a uniform radius $R>0$ such that any two elements $x_0,x_1\in\oo$ with $d_2(x_0,x_1)<R$ are joined by a unique minimal geodesic. 

For $p\ge 2$, it was shown  in \cite{odospe} that if $\oo=\oo_P=\{uPu^*:u\in \upe\}$, with $P$ an infinite self-adjoint projection of ${\cal B}({\cal H})$, then any two elements $P_0,P_1$ are joined by a minimal geodesic, which is unique if $\|P_0-P_1\|<1$.

Let us state the following partial answer to this question.

\begin{prop}
Suppose that  $\g$ is finite dimensional. If $x_0,x_1\in \oo$ satisfy $d_p(x_0,x_1)<\pi/4$, then there exists a unique minimal curve joining them, which is of the form $\delta(t)=e^{tz}\cdot x_0$, with $z\in \g_{x_0}^{\perp_p}$.
\end{prop}
\begin{proof}
Since $d_p(x_0,x_1)<\pi/4$, there exists a smooth curve $\gamma(t)\in \oo$, $t\in[0,1]$ such that $\gamma(i)=x_i$, $i=0,1$, and $L(\gamma)<\pi/4$. Then by Proposition \ref{lift}, there exists a smooth isometric lift $\Gamma(t)\in\upe$ with $\Gamma(0)=1$, $\Gamma(1)=u_1$ and $L_p(\Gamma)=L(\gamma)$. Note that $u_1\cdot x_0=x_1$. Denote by $d$ the distance 
$$
d=d_p(1, u_1 G_{x_0})=\inf\{d_p(1,u): w\in  u_1 G_{x_0}\}.
$$   
Let $w_n$ be a sequence in $u_1 G_{x_0}$ such that $d_p(1,w_n)\to d$. Since $\g$ is finite dimensional, there exists a convergent subsequence, which we still denote by $w_n$, $w_n\to w_0$. We may also suppose that $d_p(1,w_n)<\pi/4$ for all $n$. Note that $d_p(1,w_0)\le d_p(1,u_1)<\pi/4$. In other words, $w_0$ achieves the distance between $q$ and $u_1G_{x_0}$. By the convexity property of $d_p$, it is unique: if $v_0$ is another element with $d=d_p(1,v_0)$, and $\mu(t)$ is the geodesic joining $w_0$ and $v_0$, since the map $f_p(t)=d_p(1,\mu(t))^p$ is strictly convex, it follows that $v_0=w_0$. Clearly there exists $z\in\bpe_{ah}$ such that $\|z\|<\pi/4$ and $\mu(t)=e^{tz}$ is the minimal curve in $\upe$ joining $1$ and $w_0$. Then it is apparent that $\delta(t)=e^{tz}\cdot x_0$ is the unique minimal curve joining $x_0$ and $x_1$ in $\oo$. As shown before, the fact that $w_0$ is a critical point of the distance function, implies that $z\in\g_{x_0}^{\perp_p}$.
\end{proof}

\bigskip

\noindent
Esteban Andruchow and Gabriel Larotonda\\
Instituto de Ciencias \\
Universidad Nacional de Gral. Sarmiento \\
J. M. Gutierrez 1150 \\
(1613) Los Polvorines \\
Argentina  \\
e-mails: eandruch@ungs.edu.ar, glaroton@ungs.edu.ar

\bigskip

\noindent
L\'azaro Recht \\
Departamento de Matem\'atica P y A \\
Universidad Sim\'on Bol\'\i var \\
Apartado 89000\\
Caracas 1080A \\
Venezuela  \\
e-mail: recht@usb.ve

\end{document}